\documentclass[11pt]{amsart}
\usepackage{amssymb,amscd,  latexsym, graphicx, mathrsfs, enumerate}

\newcommand{\nc}{\newcommand}
\numberwithin{equation}{section}
\newtheorem{theorem}{Theorem}[section]
\newtheorem{prop}[theorem]{Proposition}
\newtheorem{importnota}[theorem]{Important Notation}
\newtheorem{prblm}[theorem]{Problem}
\newtheorem{notation}[theorem]{Notation}
\newtheorem{caution}[theorem]{Caution}
\newtheorem{remark}[theorem]{Remark}
\newtheorem{lemma}[theorem]{Lemma}
\newtheorem{construction}[theorem]{Construction}
\newtheorem{corollary}[theorem]{Corollary}
\newtheorem{example}[theorem]{Example}
\newtheorem{conclusion}[theorem]{Conclusion}
\newtheorem{triviality}[theorem]{Triviality}
\newtheorem{proto}[theorem]{Prototype Quasifibration}
\newtheorem{cauex}[theorem]{Cautionary Example}
\newtheorem{propositiondef}[theorem]{Proposition-Definition}
\newtheorem{subth}{Nuisance}[theorem]
\newtheorem{ssubth}{ }[subth]
\newtheorem{conjecture}[theorem]{Conjecture}
\newtheorem{sidest}[theorem]{Side Story}
\newtheorem{miniexample}[theorem]{Example}
\theoremstyle{definition}
\newtheorem{defin}[theorem]{Definition}

\nc\tri[1]{\begin{triviality}}
\nc\side[1]{\begin{sidest}}
\nc\conj[1]{\begin{conjecture}}
\nc\prodef[1]{\begin{propositiondef}}
\nc\prt[1]{\begin{proto}}
\nc\lem[1]{\begin{lemma}}
\nc\sblm[1]{\begin{sublemma}}
\nc\pro[1]{\begin{prop}}
\nc\thm[1]{\begin{theorem}}
\nc\cor[1]{\begin{corollary}}
\nc\dfn[1]{\begin{defin}}
\nc\sthm[1]{\begin{subth}}
\nc\exm[1]{\begin{example}}
\nc\miniexm[1]{\begin{miniexample}}
\nc\plm[1]{\begin{prblm}}
\nc\rmk[1]{\begin{remark}}
\nc\subrmk[1]{\begin{subremark}}
\nc\ntn[1]{\begin{notation}}
\nc\cau[1]{\begin{caution}}
\nc\imn[1]{\begin{importnota}}
\nc\cax[1]{\begin{cauex}}
\nc\con[1]{\begin{construction}}
\nc\ssthm[1]{\begin{ssubth}}
\nc\cnc[1]{\begin{conclusion}}
\nc\elem{\end{lemma}}
\nc\esblm{\end{sublemma}}
\nc\eside{\end{sidest}}
\nc\econj{\end{conjecture}}
\nc\eprodef{\end{propositiondef}}
\nc\eprt{\end{proto}}
\nc\ethm{\end{theorem}}
\nc\ecor{\end{corollary}}
\nc\edfn{\end{defin}}
\nc\esthm{\end{subth}}
\nc\epro{\end{prop}}
\nc\etri{\end{triviality}}
\nc\eexm{\end{example}}
\nc\eminiexm{\end{miniexample}}
\nc\ermk{\end{remark}}
\nc\subermk{\end{subremark}}
\nc\eplm{\end{prblm}}
\nc\ecau{\end{caution}}
\nc\ecax{\end{cauex}}
\nc\eimn{\end{importnota}}
\nc\entn{\end{notation}}
\nc\econ{\end{construction}}
\nc\ecnc{\end{conclusion}}
\nc\essthm{\end{ssubth}}
\newcommand{\bQ}{\overline{\mathbb{Q}}}

\newcommand{\C}{\mathbb{C}}

\newcommand{\Q}{\mathbb{Q}}

\newcommand{\Z}{\mathbb{Z}}

\newcommand{\F}{\mathbb{F}}

\newcommand{\G}{\Gamma}

\newcommand{\ds}{\displaystyle}

\newcommand{\bF}{\overline{\F}}

\newcommand{\br}{\overline{\rho}}
\newcommand{\lra}{\longrightarrow}

\newcommand{\bs}{\backslash}
\newcommand{\ve}{\varepsilon}

\title[Congruences of Siegel Eisenstein series of degree two]
{Congruences of Siegel Eisenstein series of degree two}
\author{Takuya Yamauchi}
\thanks{The author is partially supported by 
 JSPS KAKENHI Grant Number (B) No.19H01778.}
\subjclass[2010]{}

\address{Takuya Yamauchi \\
Mathematical Inst. Tohoku Univ.\\
 6-3,Aoba, Aramaki, Aoba-Ku, Sendai 980-8578, JAPAN}
\email{yamauchi@math.tohoku.ac.jp}

\begin{document}
\begin{abstract} 
In this paper we study congruences between  
Siegel Eisenstein series and Siegel cusp forms for ${\rm Sp}_4(\Z)$. 

\end{abstract}
\maketitle

\section{Introduction}The congruences for automorphic forms have been studied throughly in several settings. 
The well-known prototype seems to be one related to Ramanujan delta function and the Eisenstein series of weight 12 for ${\rm SL}_2(\Z)$. 
Apart from its own interests, various important applications, for instance Iwasawa theory,  have been made since many individual examples found before. 
In \cite{Ribet} Ribet tactfully used Galois representations for elliptic cusp forms of level one to obtain the converse of the 
Herbrand's theorem in the class numbers for cyclotomic fields. Herbrand-Ribet's theorem is now understood in the context of 
Hida theory and Iwasawa main conjecture \cite{BCG}. The notion of congruence modules is developed by many authors (cf. \cite{Hida1},\cite{Ohta} 
among others). In this paper we study the congruence primes between Siegel Eisenstein series of degree 2 and Siegel cusp forms with respect to ${\rm Sp}_4(\Z)$. To explain the main claims we need to fix the notation. Let $\mathbb{H}_2=\{Z\in M_2(\C)\ |\ {}^tZ=Z>0\}$ be the Siegel 
upper half space of degree 2 and $\G:={\rm Sp}_4(\Z)=\Bigg\{\gamma\in {\rm GL}_4(\Z)\ \Bigg|\ {}^t\gamma\begin{pmatrix}
0_2 & E_2 \\
-E_2 & 0_2
\end{pmatrix}\gamma=\begin{pmatrix}
0_2 & E_2 \\
-E_2 & 0_2
\end{pmatrix} \Bigg\}$ where $E_2$ is the identity matrix of size 2. The group $\G$ acts on $\mathbb{H}_2$ by $\gamma Z:=(AZ+B)(CZ+D)^{-1}$ for 
$\gamma=\begin{pmatrix}
A & B \\
C & D
\end{pmatrix}\in \G$ and $Z\in \mathbb{H}_2$. For such $\gamma$ and $Z$ we also define the automorphic factor $j(\gamma,Z):=\det(CZ+D)$. Put $\G_\infty=
\Bigg\{\begin{pmatrix}
A & B \\
0_2 & D
\end{pmatrix}\in \G\Bigg\}$. 
For each positive even integer $k\ge 4$ let us consider the Siegel Eisenstein series of weight $k$ with respect to $\G$:  
$$E^{(2)}_{k}(Z)=\ds\sum_{\gamma\in \G_\infty\bs \G}\frac{1}{j(\gamma,Z)^k},\ Z\in \mathbb{H}_2.$$
This series convergences absolutely and it has the Fourier expansion  
$$E^{(2)}_{k}(Z)=\ds\sum_{T\in {\rm Sym}^2(\Z)^\ast_{\ge 0}}A(T)q^T,\ q^T={\rm exp}(2\pi\sqrt{-1}{\rm tr}(TZ))$$
where ${\rm Sym}^2(\Z)^\ast_{\ge 0}$ stands for the set of semi-positive 2 by 2 symmetric matrices whose diagonal entries are integers 
and the others are half integers. Let $B_n$ be the $n$-th Bernoulli number and $B_{n,\chi}$ the generalized $n$-th  Bernoulli number 
for a Dirichlet character $\chi$. 
It is well-known (see \cite{Kau},\cite{Maass}, and Corollary 2, p. 80 of \cite{EZ}) that $A_T$ is explicitly given as follows: 
\begin{itemize}
\item if rank$(T)=0$, hence $T=0_2$, then $A(0_2)=1$; 
\item if rank$(T)=1$, there exists $S\in {\rm GL}_2(\Z)$ such that ${}^t STS=\begin{pmatrix} n& 0\\ 0 &0\end{pmatrix}$ for some 
positive integer $n$. 
Then $A(T)=2\ds\frac{\sigma_{k-1}(n)}{\zeta(1-k)}=-\ds\frac{2k}{B_{k}}\sigma_{k-1}(n)$ where 
$\zeta(s)$ stands for the Riemann zeta function;
\item if rank$(T)=2$, 
\begin{equation}\label{st}
A(T)=2\frac{L(2-k,\chi_T)}{\zeta(1-k)\zeta(3-2k)}S_T=-4k\frac{B_{k-1,\chi_T}}{B_{k}B_{2k-2}},\ S_T:=\prod_{q|\det(2T)}F_q(T;q^{k-3})
\end{equation}
where $F_q(T,X)\in \Z[X]$ is the Siegel series for $T$ at any prime $q$ dividing $\det(2T)$ and $\chi_T$ is the quadratic character 
associated to the quadratic extension $\Q(\sqrt{-\det(2T)})/\Q$. Notice that $S_T\in \Z$ and it is known that $S_T=1$ for which $-\det(2T)$ is a fundamental discriminant. 
\end{itemize}
For each rational prime $p$ we fix the embeddings $\bQ\hookrightarrow \bQ_p,\ \bQ\hookrightarrow \C$ and an isomorphism $\C\simeq \bQ_p$ 
which is compatible with those embeddings. For any extension $K$ of $\Q$ or $\Q_p$ we denote by $\mathcal{O}=\mathcal{O}_K$ the ring of integers of $K$ 
and $\mathcal{P}=\mathcal{P}_K$ 
a prime ideal above $p$. We often drop the superscript ``$K$" if it is obvious from the context. 
For a Siegel modular form $F$ of degree 2  with the Fourier expansion  
$F(Z)=\ds\sum_{T\in {\rm Sym}^2(\Z)^\ast_{\ge 0}}A_F(T)q^T$ we say $F$ is defined over $\mathcal{O}$ if $A(T)\in \mathcal{O}$ 
 for any $T\in {\rm Sym}^2(\Z)^\ast_{\ge 0}$ and we often say $F$ is $\mathcal{O}$-integral if $F$ is defined over $\mathcal{O}$ for some $\mathcal{O}$.  
 For two integral Siegel modular forms $F,G$ of degree 2 with the Fourier expansions 
$F(Z)=\ds\sum_{T\in {\rm Sym}^2(\Z)^\ast_{\ge 0}}A_F(T)q^T,\ G(Z)=\ds\sum_{T\in {\rm Sym}^2(\Z)^\ast_{\ge 0}}A_G(T)q^T,\ 
A_F,A_G\in \mathcal{O}$ for some $\mathcal{O}$ we say 
$F$ is congruent to $G$ modulo $\mathcal{P}$ or simply we write $F\equiv G$ mod $\mathcal{P}$ if $A_F(T)\equiv A_G(T)$ mod $\mathcal{P}$ for any $T\in {\rm Sym}^2(\Z)^\ast_{\ge 0}$. We also say  $F$ is a non-trivial $\mathcal{O}$-integral 
Siegel modular form if 
$F\not\equiv 0$ mod $\mathcal{P}$. Hence there exists at least one $T\in {\rm Sym}^2(\Z)^\ast_{\ge 0}$ 
such that $A_F(T)\in \mathcal{O}^\times$. This implies that if further $F$ is a Hecke eigen form, then all Hecke eigenvalues are 
$\mathcal{O}$-integral.    
For two non-trivial $\mathcal{O}$-integral Siegel modular forms $F,G$, we write $F\equiv_{{\rm ev}}G$ mod $\mathcal{P}$ if 
they share all Hecke eigenvalues modulo $\mathcal{P}$. Obviously if $F\equiv G$ mod $\mathcal{P}$ and either of 
$F,G$ is non-trivial, then $F\equiv_{{\rm ev}}G$ mod $\mathcal{P}$. 

We normalize $E_k$ by 
\begin{equation}\label{ne}
G_k(Z):=\frac{1}{2}\zeta(1-k)\zeta(3-2k)E_k(Z).
\end{equation}
Now we are ready to state our main results:
\begin{theorem}\label{main1}
Let $k$ be even positive integer and $p\ge 7$ be a prime number such that $p-1>2k-2$. 
Assume that ${\rm ord}_p(B_{2k-2})>0$ and ${\rm ord}_p(B_{k-1,\chi_T})=0$ for some $T\in {\rm Sym}^2(\Z)^\ast_{> 0}$. 
Then there exists a non-trivial $\mathcal{O}$-integral Hecke eigen Siegel cusp form $F$ defined over  
 $\mathcal{O}$ for some finite extension $K/\Q$ such that $F\equiv_{{\rm ev}}G_k$ mod $\mathcal{P}$  for a prime ideal $\mathcal{P}$ of $\mathcal{O}$ dividing $p$. 
\end{theorem} 
As explained later the condition $p-1>2k-2$ implies that $A_{G_k}(T)\in \Z_p$ for any $T\in {\rm Sym}^2(\Z)^\ast_{\ge 0}$. 
It is natural to ask an explicit form of $F$ in the theorem above. 
In fact, we can realize it as a Saito-Kurokawa lift.    
\begin{theorem}\label{main-SK}Let $k$ be even positive integer and $p\ge 7$ be a prime number such that $p-1>2k-2$. 
Assume that ${\rm ord}_p(B_{2k-2})>0$ and ${\rm ord}_p(B_{k-1,\chi})=0$ 
for some quadratic character $\chi$ of $G_\Q$.  
Then there exist  $\mathcal{O}_K$-integral Saito-Kurokawa lifts $F_i=SK(f_i),\ (1\le i \le t)$ for some $f_i\in S_{2k-2}({\rm SL}_2(\Z)), (1\le i \le t)$ and for some 
finite extension $K/\Q$ such that $\ds\sum_{i=1}^tF$ is congruent to $G_k$ modulo a prime $\mathcal{P}$ of $\mathcal{O}$ dividing $p$. 
\end{theorem}
\begin{remark}The number of Saito-Kurokawa lifts in the claim above are exactly given as one of Hecke eigen cusp forms in $S_{2k-2}({\rm SL}_2(\Z))$ which are congruent to $G^{(1)}_{2k-2}$ modulo $\mathcal{P}$ (see the proof of Proposition \ref{sc}). 
Here   $G^{(1)}_{2k-2}$ is the Eisenstein series of weight $2k-2$ for ${\rm SL}_2(\Z)$ with the constant term $\zeta(3-2k)$.   
\end{remark}

We also study a weak version of the converse of Theorem \ref{main1}. Let $h^+_p$ be the plus part of the ideal class group of $\Q(\zeta_p)$ where 
$\zeta_p$ is a primitive $p$-th root of unity. 
\begin{theorem}\label{main2} Let $k$ be even positive integer and $p\ge 7$ be a prime number such that $p-1>2k-2$.   
If there exists a non-trivial $\mathcal{O}_K$-integral Hecke eigen Siegel cusp form $F$ for some number field $K$ which is congruent to $G_{k}$ modulo some prime ideal above $p$, 
then ${\rm ord}_p(B_kB_{2k-2}h^+_p)>0$. In particular, if Vandiver conjecture is true (hence $p\nmid h^+_p$) and ${\rm ord}_p(B_k)=0$, then it follows that ${\rm ord}_p(B_{2k-2})>0$. 
\end{theorem}
The following claim is an obvious conclusion of the theorem above and it explains Siegel cusp forms tends to 
be residually far from Siegel Eisenstein series:  
\begin{corollary}Let $k$ be an even positive integer and $p\ge 7$ be a prime number such that $p-1>2k-2$. 
Assume that ${\rm ord}_p(B_kB_{2k-2}h^+_p)=0$. Then there is no Hecke eigen cusp form which is congruent to 
$G_k$ modulo a prime over $p$. In particular if Vandiver's conjecture is true, then the claim is true under the condition that 
${\rm ord}_p(B_kB_{2k-2})=0$ which is depending only on $k$.  
\end{corollary}
Some experts may think of the case when $k$ is odd. In fact we can study the congruence of Siegel cusp forms of 
odd weights by using mod $p$ Galois representations instead of Siegel Eisenstein series. 
\begin{theorem}\label{main3} Let $k$ be an odd positive integer and $p\ge 7$ be a prime number such that $p-1>2k-2$. 
Let $F$ be a non-trivial $\mathcal{O}_K$-integral Hecke eigen Siegel cusp form whose mod $p$ Galois representation 
$\br_{F,p}$ 
is completely reducible.   
Then ${\rm ord}_p(B_{k-1}B_{2k-2}h^+_p)>0$. In particular, if Vandiver conjecture is true and ${\rm ord}_p(B_{k-1})=0$, then it follows that ${\rm ord}_p(B_{2k-2})>0$. 
\end{theorem}

In general any genuine Siegel cusp form $F$ of weight $k\ge 3$ can not be congruent to Siegel Eisenstein series modulo $p$ for 
all but finitely many prime $p$ by Theorem B of \cite{Weiss}. Therefore the novelty of this paper is to 
detect congruence primes related to Siegel Eisenstein series in terms of weight and the results show 
they are completely depending on the weight $k$ provided if Vandiver's conjecture is true.      

As for Theorem \ref{main1}, a key is to use some geometric techniques to find such a Siegel cusp form out 
where we require $p\ge 7$ to connect the classical Siegel modular forms with the geometric Siegel modular forms because 
our forms is of level one, hence there is no direct way to define them in terms of geometry. 
For the claim of Theorem \ref{main2}, we use some techniques in Galois representations of Siegel modular forms which are borrowed from Ribet' argument in \cite{Ribet}. They are well-known for experts. It is natural to ask any other property for the Siegel cusp form $F$ in Theorem \ref{main1}. 
Since we focus only on the level one, there are exactly two kinds of Siegel cusp forms of weight $k\ge 4$ 
(in fact, it should be ``$k\ge 10$" to exist a non-zero cusp form). One is so called 
a Saito-Kurokawa lift and the other one is called a genuine form so that its $p$-adic Galois representation is 
(absolutely) irreducible for any prime $p$. As is done by Brown \cite{Brown} in addition to the condition of Theorem \ref{main1} if 
we further impose some kind of divisibility conditions for $L$-values of $f$,  then we may have a genuine form $G$ which is 
congruent to $G_k$. However it may very rarely happen as long as we assume the level to be one though we do not pursue it here.  

This paper is organized as follows. In Section 2 we summarize the basic properties of Siegel modular forms and its geometric counterpart. 
In Section 3 we discuss the $p$-integrality of the Fourier coefficients of Siegel Eisenstein series. 
We devote Section 4, Section 5, and Section 6 for proofs of the main theorems. 
In the last section we give a few examples. 

\textbf{Acknowledgments.} The author would like to thank Florian Herzig, Can-Ho Kim, Iwao Kimura, Ariel Weiss for helpful comments and many valuable discussions. 
In particular, Herzig and Weiss kindly informed the author an error of \cite{CG} for 
the irreducibility of mod $p$ Galois representations attached to RAESDC automorphic representations.    
This work started when the author visited Pavel Guerzhoy at University of Hawaii. 
The author would also like to thank him for valuable discussions and the university for the incredible hospitality. 

\section{The basics of Siegel modular forms}
In this section we refer to \cite{Taylor-thesis}, \cite{vG} for the classical Siegel modular forms and Chapter V of \cite{CF} for the geometric 
Siegel modular forms. 

Let us fix the notation. Put $J=\begin{pmatrix} 0_2& I_2\\ -I_2 &0_2\end{pmatrix}$. 
For any commutative ring $R$ we define ${\rm GSp}_4(R)=\{g\in M_4(R)\ |\ {}^t XJX=\nu(g)J,\ \nu(g)\in R^\times \}$ where $\nu$ is 
called the similitude character. Put ${\rm Sp}_4(R)={\rm Ker}(\nu:{\rm GSp}_4(R)\lra R^\times)$. Let $P$ be the Siegel parabolic subgroup of $GSp_4$ given by $P=\Bigg\{
\begin{pmatrix} A& B\\ 0_2 &D\end{pmatrix}\in {\rm GSp}_4 \Bigg\}=MN$ where $M=
\Bigg\{m(A,\nu):=
\begin{pmatrix} A& 0_2\\ 0_2 &\nu\cdot {}^tA^{-1}\end{pmatrix}\ \Big|\ A\in GL_2,\ \nu\in GL_1\in {\rm GSp}_4 \Bigg\}$ is the Levi factor and 
$N=\Bigg\{n(B):=
\begin{pmatrix} 1_2& B\\ 0_2 &1_2\end{pmatrix}\ \Big|\ {}^tB=B\in M_2 \Bigg\}$ is the unipotent radical. 
For any integer $M\ge 3$ we denote by $\G(M)$ the principal congruence subgroup of level $M$ which is the subgroup of 
$\G:={\rm Sp}_4(\Z)$ consisting of all elements congruent to $1_2$ modulo $M$. 
Let $\G_0$ be a congruence subgroup of $\G$. 
Since $M\ge 3$ for any positive integer $k$, the group $\G_0\cap \G(M)$ neat and therefore any commutative 
ring $R$  we can define $\mathcal{M}_k(\G_0\cap \G(M),R)$ (resp. $\mathcal{S}_k(\G_0\cap \G(M),R)$) which is the space of the geometric Siegel modular forms 
(resp. Siegel cusp forms) of weight $k$ with 
respect to $\G_0\cap \G(M)$. For a positive integer $N$ let us focus on the group $\G_0=\G_0(M)=
\Bigg\{\begin{pmatrix}
A & B \\
C & D
\end{pmatrix}\in \G\ \Bigg|\ C\equiv 0_2\ {\rm mod}\ M \Bigg\}$ and any prime $M=\ell\ge 3$ which does not divide $N$. Let $\zeta_\ell$ be 
a primitive root of unity. 
It is well-known that $\mathcal{M}_k(\G_0(N)\cap \G(\ell),\Z[\frac{1}{\ell N},\zeta_\ell])\otimes \C\simeq M_k(\G_0(N)\cap \G(\ell))$ 
preserving cusp forms  
where the latter space consists of all classical Siegel modular forms of weight $k$ with respect to $\G_0(N)\cap \G(\ell)$. 
The finite group $G_\ell={\rm Sp}_4(\Z/\ell\Z)\simeq \G/\G(\ell)$ naturally acts on $\mathcal{M}_k(\G_0(N)\cap \G(\ell),\Z[\frac{1}{\ell N},\zeta_\ell])$ by Harris's result \cite{Harris}. 
For any $\Z[\frac{1}{\ell N},\zeta_\ell]$-algebra $R$ 
put $\mathcal{M}_k(\G_0(N)_\ell,R):=\mathcal{M}_k(\G_0(N))\cap \G(\ell),R)^{G_\ell}$ and 
$\mathcal{S}_k(\G_0(N)_\ell,R):=\mathcal{S}_k(\G_0(N)\cap \G(\ell),R)^{G_\ell}$. For each rational prime $p$ we fix embeddings $\bQ_p\hookrightarrow \bQ_p,\ \bQ\hookrightarrow \C$ and an isomorphism $\C\simeq \bQ_p$ 
which is compatible with the previous embeddings. We start with the following facts. 
\begin{prop}\label{b1} Let $p\ge 3$ be a rational prime which does not divide $\ell N |G_\ell|$ where $|G_\ell|$ stands for the cardinality of $G_\ell$. 
Then the followings hold: 
\begin{enumerate}
\item For any finite extension $K$ of $\Q_p(\zeta_\ell)$, $\mathcal{M}_k(\G_0(N)_\ell,\mathcal{O}_K)\otimes_{\mathcal{O}_K}\C\simeq 
M_k(\G_0(N))$ preserving cusp forms where the latter space consists of all Siegel modular forms of weight $k$ with respect to $\G_0(N)$.    
\item If $k\ge 4$ then the natural reduction map $\mathcal{S}_k(\G_0(N)\cap \G(\ell),W(\bF_p))\lra 
\mathcal{S}_k(\G_0(N))\cap \G(\ell),\bF_p)$ is surjective. Here $W(\bF_p)$ stands for the ring of Witt vectors for $\bF_p$. 
\item If $k\ge 4$ then the natural reduction map $\mathcal{S}_k(\G_0(N)_\ell,W(\bF_p))\lra 
\mathcal{S}_k(\G_0(N)_\ell,\bF_p)$ is surjective.  
\end{enumerate}      
\end{prop}
\begin{proof}For the first statement, since $\G_0(N)\supset \G_0(N))\cap \G(\ell)$, clearly we have $M_k(\G_0(N))\subset 
M_k(\G_0(N))\cap \G(\ell))\simeq \mathcal{M}_k(\G_0(N)\cap \G(\ell),\Z[\frac{1}{\ell N},\zeta_\ell])\otimes \C$. We observe more finer structure 
of this inclusion. 

For the 0-dimensional cusp associated to the standard Siegel parabolic subgroup in $Sp_4$ we define the Fourier expansion map  
$$\mathcal{M}_k(\G_0(N)\cap \G(\ell),R)\lra R_\ell,\ F\mapsto F(q_\ell):=\sum_{T\in {\rm Sym}^\ast_2(\Z)_{\ge 0}}A_F(T)q^T_\ell$$
for any $\Z[\frac{1}{\ell N},\zeta_\ell]$-algebra $R$ where  
$q^T_\ell:=q^{\frac{a}{\ell}}_{11}\cdot q^\frac{b}{\ell}_{12}\cdot q^\frac{c}{\ell}_{22}$ for $T=\left(\begin{array}{cc}
a & \frac{b}{2}\\
\frac{b}{2} & c 
\end{array}
\right)$ with formal parameters $q^{\frac{1}{\ell}}_{11},q^{\frac{1}{\ell}}_{12},q^{\frac{1}{\ell}}_{22}$ and $R_\ell=R[q^{\pm \frac{1}{\ell}}_{12}][[q^{\frac{1}{\ell}}_{11},q^{\frac{1}{\ell}}_{22}]]$. 
By the proof of Lemma 2.1 of \cite{Taylor-thesis} and $q$-expansion principle (cf. Proposition 1.8-(iv) in p.146 of \cite{CF}), 
for any finite extension $K$ of $\Q_p(\zeta_\ell)$ if we define $M_k(\G_0(N)\cap \G(\ell),\mathcal{O}_K)$ by the subspace of 
$M_k(\G_0(N)\cap \G(\ell))$ consisting of all forms $F$ such that $F(q)$ is defined over $\mathcal{O}_K$. Similarly we can also define 
$M_k(\G_0(N),\mathcal{O}_K)$ by using $q$-expansion principle. Notice that $p\not||G_\ell|$ and this implies that 
$M_k(\G_0(N),\mathcal{O}_K)=M_k(\G_0(N)\cap \G(\ell),\mathcal{O}_K)^{G_\ell}$. Hence we have 
$$M_k(\G_0(N))\simeq M_k(\G_0(N),\mathcal{O}_K)\otimes_{\mathcal{O}_K} \C=M_k(\G_0(N)\cap \G(\ell),\mathcal{O}_K)^{G_\ell}\otimes_{\mathcal{O}_K} \C=\mathcal{M}_k(\G_0(N)),\mathcal{O}_K)\otimes_{\mathcal{O}_K}\C .$$

For the second claim we may apply Corollary 4.3 of \cite{LS} with $d=2$ and $k_C=3$ in their notation in conjunction with our setting. Note that our group 
$\G_0(N)\cap \G(\ell)$ is obviously neat.  

For the third claim, by the above claim we have the exact sequence 
$$0\lra \mathcal{S}_k(\G_0(N)\cap \G(\ell),W(\bF_p))\stackrel{\times p}{\lra}\mathcal{S}_k(\G_0(N)\cap \G(\ell),W(\bF_p))\lra 
\mathcal{S}_k(\G_0(N))\cap \G(\ell),\bF_p)\lra 0$$
as $G_\ell$-modules. 
By taking the Galois cohomology the obstruction to fail the surjectivity in question is aggregated to 
$H^1(G_\ell, \mathcal{S}_k(\G_0(N)\cap \G(\ell),W(\bF_p)))$. Since $p$ does not divide $|G_\ell|$, it vanishes hence we have the claim.

\end{proof}
Let us keep the notation being as above. 
For any congruence subgroup $\G$ of ${\rm Sp}_4(\Z)$ we denote by ${\rm Cusp}^0(\G)$ the set of 
0-dimensional cusps for $\G$ and is explicitly given by 
$${\rm Cusp}^0(\G)=\G\bs {\rm Sp}_4(\Q)/P(\Q)\simeq \G\bs {\rm Sp}_4(\Z)/P(\Z)$$
(see Section 4 of \cite{Na}). 
We define the projection $\pi:P':=P\cap Sp_4\lra P\cap M\simeq GL_2$ by sending $m(A,1)n(B)$ to $A$. 
For each representative $\gamma\in {\rm Cusp}^0(\G)$ put $\G_\gamma:=\pi(P'(\Z)\cap \gamma^{-1}\G \gamma)\subset {\rm GL}_2(\Z)$. 

Let $R$ be a $\Z[\frac{1}{\ell N},\zeta_\ell]$-algebra which is domain. 
According to Chapter V of \cite{CF} we consider the Siegel operator $\Phi_{c,R}:\mathcal{M}_k(\G_0(N)\cap \G(\ell),R)\lra 
\mathcal{M}^{(1)}_k((\G_0(N)\cap \G(\ell))_c,R)$ for any  $c\in {\rm Cusp}^0(\G_0(N)\cap \G(\ell))$ (see  lines 13-14 in p.144 of \cite{CF}). Here  $\mathcal{M}^{(1)}_k((\G_0(N)\cap \G(\ell))_c,W(\F_p))$ stands for the geometric elliptic modular forms of weight $k$ 
with respect to $(\G_0(N)\cap \G(\ell))_c$. Then we have $\mathcal{S}_k(\G_0(N)\cap \G(\ell),\mathcal{O}_K)=
\ds\bigcap_{c\in {\rm Cusp}^0(\G_0(N)\cap \G(\ell))}{\rm Ker}(\Phi_{c,\mathcal{O}_K})$ by which 
lines 13-14 in p.144 of \cite{CF} explain.   

Let us focus on $\ell=3$ and recall the notation that $\mathcal{M}_k({\rm Sp}_4(\Z)_3,R):=
\mathcal{M}_k({\rm Sp}_4(\Z)\cap \G(3),R)$ and $\mathcal{S}_k({\rm Sp}_4(\Z)_3,R):=
\mathcal{S}_k({\rm Sp}_4(\Z)\cap \G(3),R)$ for any $\Z[\frac{1}{3},\zeta_3]$-algebra $R$.  
\begin{prop}\label{b2}Keep the notation being as above. Assume that $p\ge 7$. 
Then it holds that 
\begin{enumerate}
\item for any finite extension $K/\Q_p(\zeta_3)$ with the residue field $\F$, the reduction map $S_k({\rm Sp}_4(\Z),\mathcal{O}_K)\otimes \F\lra 
\mathcal{S}_k({\rm Sp}_4(\Z)_3,\F)$ is surjective and for any Hecke eigen form $H$ in $\mathcal{S}_k({\rm Sp}_4(\Z)_3,\F)$ 
there exists a non-trivial $\mathcal{O}_K$-integral Hecke eigen form $G$ in $S_k({\rm Sp}_4(\Z),\mathcal{O}_K)$ such that the reductions of all Hecke eigen values of $G$ coincide with those of $H$;  
\item $\mathcal{S}_k({\rm Sp}_4(\Z)_3,\F)=\{F\in \mathcal{M}_k({\rm Sp}_4(\Z)_3,\F)\ |\ A_F(T)=0 {\rm \ for\ any\ degenerate}\ T\in {\rm Sym}_2(\Z)^\ast_{\ge 0} \}$. 
\end{enumerate}  
\end{prop}
\begin{proof} Notice that $|G_3|=|{\rm Sp}_4(\Z/3\Z)|=80=2^4\cdot 5$. Hence $p\not| |G_3|$. Then the first claim follows from 
Proposition \ref{b1}-(3). The latter part follows from Lemma 6.11 of \cite{DS}

For the second claim, 
let $F$ be an element in the right hand side. 
Then  $F$ can be also regarded as an element in $\mathcal{M}_k(\G(3),\F)$. 
By assumption on $F$ and using the injectivity of the $q$-expansion (see Proposition 1.8-(ii) of \cite{CF}), we see that $\Phi_{1_4,\F}(F)=0$ for the trivial element $1_4\in {\rm Cusp}^0(\G(3))$ which corresponds to the standard Siegel parabolic. 
Since $F$ is of level one, $G_3$ trivially acts on $F$ . Hence  $\Phi_{c,\F}(F)=0$ for any $c\in {\rm Cusp}^0(\G(3))$ and implies 
$F\in \mathcal{S}_k(\G(3),\F)$. Therefore we have $F\in \mathcal{S}_k({\rm Sp}_4(\Z)_3,\F)$. The other implication is obvious and omitted. 
 
\end{proof}

\begin{corollary}\label{imp}
Let $F$ be a Hecke eigen form in $M_k({\rm Sp}_4(\Z))$ whose Fourier expansion $F(Z)=\ds\sum_{T\in {\rm Sym}_2(\Z)^\ast_{\ge 0}}A_F(T)q^T$ is defined over 
$\mathcal{O}_K$ for some finite extension $K/\Q_p(\zeta_3)$ with $p\ge 7$. Let $\mathcal{P}$ be the maximal ideal of $\mathcal{O}_K$. 
Assume that $A_F(T)\in \mathcal{P}$ for any degenerate $T$ and there exists at least one $T \in {\rm Sym}_2(\Z)^\ast_{> 0}$ 
such that $A_F(T)\in \mathcal{O}^\times_K$. Then there exists a non-trivial Hecke eigen $\mathcal{O}_K$-integral cusp form $G$ in $S_k({\rm Sp}_4(\Z))$ such that $F\equiv_{{\rm ev}} G$ mod $\mathcal{P}$.  
\end{corollary}
\begin{proof}
Put $\F=\mathcal{O}_K/\mathcal{P}$. 
We first apply  Proposition \ref{b2}-(ii) to confirm that $\overline{F}:=F$ mod $\mathcal{P}$ belongs to 
$\mathcal{S}_k({\rm Sp}_4(\Z)_3,\F)$ as a non-zero element. Then by Proposition \ref{b2}-(1) there exists such a form $H$. 
\end{proof}

\section{Siegel Eisenstein series} 
Let $k\ge 4$ be an even positive integer. 
Recall the normalized Eisenstein series $G_k \in M_k({\rm Sp}_4(\Z))$ defined in (\ref{ne}) whose Fourier expansion is given explicitly by 
\begin{itemize}
\item if rank$(T)=0$, hence $T=0_2$, then $A_{G_k}(0_2)=\ds\frac{1}{2}\zeta(1-k)\zeta(3-2k)=\frac{B_kB_{2k}}{2k(2k-2)}$; 
\item if rank$(T)=1$, there exists $S\in {\rm GL}_2(\Z)$ such that ${}^t STS=\begin{pmatrix} n& 0\\ 0 &0\end{pmatrix}$ for some 
positive integer $n$. 
Then $A_{G_k}(T)=\frac{1}{2}\zeta(3-2k)\sigma_{k-1}(n)=-\ds\frac{B_{2k-2}}{2(2k-2)}\sigma_{k-1}(n)$ for $n\ge 1$;
\item if rank$(T)=2$, 
$A_{G_k}(T)=L(2-k,\chi_T)S_T=-\ds\frac{B_{k-1,\chi_T}}{k-1}S_T$. Note that $S_T\in \Z$ and $S_T=1$ for which $-\det(2T)$ is a fundamental discriminant. 
\end{itemize}
It is well-known that $G_k$ is a Hecke eigen form (cf. \cite{Walling}). 

\begin{prop}\label{int} Let $p$ be a rational prime satisfying $p-1>2k-2$. Then $A_{G_k}(T)\in \Z_p$ for any 
$T\in {\rm Sym}^\ast_2(\Z)_{\ge 0}$.  
\end{prop}
\begin{proof}By Von Staudt-Clausen theorem we see both of ${\rm ord}_p(\zeta(1-k))$ and  ${\rm ord}_p(\zeta(3-2k))$ are non-negative. 
Since $p-1\not|2k-2$ by assumption, it follows from Carlitz's theorem \cite{Ca} (see also \cite{SUZ}) that 
${\rm ord}_p(L(2-k,\chi_T))\ge 0$. Putting everything together we have the claim. 
\end{proof}

We are now ready to prove our first main theorem. 
\begin{proof}(A proof of Theorem \ref{main1}) By Proposition \ref{int} and the assumption, the conditions for $G_k$ in 
Corollary \ref{imp} are fulfilled and hence we have the claim.    
\end{proof}

\section{Saito-Kurokawa lifts} In this section we give a proof of Theorem \ref{main-SK}. In what follows we assume that $k$ is 
even (cf. see \cite{Schmidt} for a reason why we assume that). 
Let us consider the Kohnen's plus space  $M^+_{k-\frac{1}{2}}(\G^{(1)}_0(4))$ (resp. $S^+_{k-\frac{1}{2}}(\G^{(1)}_0(4))$) which corresponds to $M_{2k-2}({\rm SL}_2(\Z))$ 
(rsp. $S_{2k-2}({\rm SL}_2(\Z))$) under Shimura correspondence (cf. \cite{Kohnen}). For $f\in M^+_{k-\frac{1}{2}}(\G^{(1)}_0(4))$ we denote by $Sh(f)$ the image of $f$ under the correspondence. 
Let $H_{k-\frac{1}{2}}=\ds\sum_{n\ge 0}H(k-1,n)q^n$ be the Cohen's Eisenstein series of weight $k-\frac{1}{2}$ whose constant term is 
$\zeta(3-2k)$ (cf. \cite{Cohen},\cite{DG})  and $G^{(1)}_{\ell}$ the Eisenstein series of 
weight $\ell$ with respect to ${\rm SL}_2(\Z)$ with the constant term $\zeta(1-\ell)$.  

\begin{prop}\label{sc} Keep the assumption in Theorem \ref{main-SK}. Then there exist  
$\mathcal{O}_K$-integral Hecke eigen cusp forms $h_1,\ldots,h_t$ in $S^+_{k-\frac{1}{2}}(\G^{(1)}_0(4))$ for some 
number field $K$ such that  $H_{k-\frac{1}{2}}\equiv \ds\sum_{i=1}^t h_i$ modulo a prime $\mathcal{P}_K$ dividing $p$ of $\mathcal{O}_K$. 
Further each $\phi_i:=Sh(h_i)$ ($1\le i \le t$) is congruent to $G^{(1)}_{2k-2}$ modulo $\mathcal{P}_K$. 
\end{prop}
\begin{proof}By assumption, $H(k-1,|D|)=-\ds\frac{B_{k-1},\chi_D}{k-1}$ is a $p$-adic unit where $\chi_D$ is the quadratic character with a fundamental discriminant $D<0$ and hence $H_{k-\frac{1}{2}}\not\equiv 0$ mod $\mathcal{P}_K$. 
Therefore the claim follows from Theorem 4 of \cite{DG}.  
\end{proof}

\begin{proof}(A proof of Theorem \ref{main2}.) As discussed in Section 6 of \cite{Ikeda}, 
$$H(k-1,n)=L(2-k,\chi_{-n})\ds\prod_{p|n}p^{{\rm ord}_p(f_n)(k-\frac{3}{2})}\Psi_p(n,p^{k-\frac{3}{2}})$$ if $n>0$ and $n\equiv 0,1$ mod 4 and 
$H(k-1,n)=0$ otherwise but $n>0$. Here $\chi_{-n}$ is the quadratic character corresponding to $\Q(\sqrt{-n})$. 
Here $\Psi_p(n,X)$ is defined in p.647 of \cite{Ikeda} and $n=d_n f^2_n$ is the product of a fundamental discriminant $d_n$ and the remaining square factor. By definition it is easy to see that 
there exists a polynomial $S_{n,p}(X)$ in $\Z[X]$ such that $p^{{\rm ord}_p(f_n)(k-\frac{3}{2})}\Psi_p(n,p^{k-\frac{3}{2}})=S_{n,p}(1+p^{2k-3})$ for any $k$. 
Further $$S_T=\ds\prod_{p|\det(2T)}p^{{\rm ord}_p(f_{\det(2T)})(k-\frac{3}{2})}\Psi_p(\det(2T),p^{k-\frac{3}{2}})=
\ds\prod_{p|\det(2T)}S_{\det(2T),p}(1+p^{2k-3})$$ for any $T\in {\rm Sym}^\ast_2(\Z)_{>0}$ where $S_T$ is in 
the equation (\ref{st}). 

By Propsotion \ref{sc} there exist forms $h_i=\ds\sum_{n>0}a_{h_i}(n)q^n$ and 
$\phi_i:=Sh(h_i)=\ds\sum_{n>0}a_{\phi_i}(n)q^n$ for $i=1,\ldots,t$ in the claim. 
Then by using Theorem 3.2 of \cite{Ikeda} the formal series 
$$SK(f_i):=\ds\sum_{T\in {\rm Sym}^\ast_2(\Z)_{>0}}a_{h_i}(\det(2T))\prod_{p|\det(2T)}S_{\det(2T),p}(a_{\phi_i}(p))q^T$$
is a Hecke eigen cusp form in $S_k({\rm Sp}_4(\Z))$. On the other hand, the above forms satisfy 
$\ds\sum_{i=1}^ta_{h_i}(n)\equiv H(k-1,n)$ and  $a_{\phi_i}(n)\equiv a_{G^{(1)}_{2k-2}}(n)$ modulo 
$\mathcal{P}_K$ for $i=1,\ldots,t$. Notice that $$A_{\widetilde{G}_{k}}(T)=H(k-1,\det(2T))S_{\det(2T)}$$ 
for any $T\in {\rm Sym}^\ast_2(\Z)_{>0}$. 
Hence we have 
\begin{eqnarray}
\ds\sum_{i=1}^t A_{SK(f_i)}(T)&\equiv & H(k-1,\det(2T))\prod_{p|\det(2T)}S_{\det(2T),p}( a_{G^{(1)}_{2k-2}}(p))\nonumber \\ 
&=& H(k-1,\det(2T))\prod_{p|\det(2T)}S_{\det(2T),p}(1+p^{2k-3})\\
&=& A_{\widetilde{G}_{k}}(T)\ {\rm  modulo}\  \mathcal{P}_K. \nonumber
\end{eqnarray}
This gives us the claim. 
\end{proof}

\section{Galois representations} In this section the readers are supposed to be familiar with the basics of Galois 
representations. 
We only focus on $S_k({\rm Sp}_4(\Z))$ for any $k\ge 10$. Any Hecke eigen Siegel cusp form in that space is either 
a Saito-Kurokawa lift (some experts say it a Maass form) or a genuine form. The latter form is just defined not to be Saito-Kurokawa lifts 
(cf. \cite{Schmidt}). 
The weight $k$ has to be even for Siegel cusp forms to be Saito-Kurokawa lifts since it is of level one (\cite{Schmidt} again).
Any genuine Hecke eigen Siegel cusp form of level one never comes from any functorial 
lift from small groups, that is a Saito-Kurokawa lift, an endoscopic lift, a base change lift, and a symmetric cubic lift (see \cite{KWYII}). 
The middle two lifts are excluded since the form is of level one. There is no symmetric cubic lifts to Siegel cusp forms of 
scalar weight other than three (cf. \cite{RS}). 
For any genuine Hecke eigen Siegel cusp form $F$ in $S_k({\rm Sp}_4(\Z))$ and any prime $p$ one can associate $p$-adic Galois representation 
\begin{equation}\label{galois-const}
\rho_{F,p}:G_\Q:={\rm Gal(\bQ/\Q)}\lra {\rm GSp}_4(\bQ_p)
\end{equation}
which satisfies several nice properties (see Section 3 of \cite{KWYII}).  
Let $\iota:GSp_4\hookrightarrow GL_4$ be the natural inclusion. Note that $\iota\circ \rho_{F,p}:G_\Q\lra {\rm GL}_4(\bQ_p)$ 
is irreducible (see Section 3 of \cite{KWYII} again). 
There exists a finite extension $K/\Q_p$ such that that $A_F(T)\in \mathcal{O}_K$ for all $T\in {\rm Sym}_2(\Z)^\ast_{>0}$. 
We can find $T\in {\rm Sym}_2(\Z)^\ast_{>0}$ such that $A_F(T)$ is a unit. Hence $F$ is a non-trivial $\mathcal{O}_K$-integral 
Siegel cusp form. By enlarging $K$ if necessary we may assume that $\iota\circ \rho_{F,p}$ takes the values in ${\rm GL}_4(K)$. 
Put $V_K:=K^4$ and regard it as a representation space of $\iota\circ \rho_{F,p}$. for any $G_\Q$-stable $\mathcal{O}_K$-lattice $T$ in $V_K$ 
we denote by $(\iota\circ \rho_{F,p},T)$ the corresponding representation space or equivaently the corresponding Galois representation 
$\iota\circ \rho_{F,p}:G_\Q\lra {\rm Aut}_{\mathcal{O}_K}(T)\simeq GL_4(\mathcal{O}_K)$ where the latter isomrphism is depending of a choice of a 
$\mathcal{O}_K$-basis of $T$. We also denote by $\iota\circ \br_{F,p}$ the reduction modulo the maximal prime ideal $\mathcal{P}_K$ of $\mathcal{O}_K$  and by 
$(\iota\circ \br_{F,p},\overline{T})$ the corresponding representation space where $ \overline{T}:=T\otimes_{\mathcal{O}_K}\mathcal{P}_K$. 

\begin{theorem}\label{galois}Keep the notation being as above. Assume that $p-1>2k-2$. Assume that $F$ is congruent to $G_k$ modulo $\mathcal{P}_K$. 
Then there exists a $\mathcal{O}_K$-lattice $T$ of $\iota\circ \rho_{F,p}$ such that 
\begin{enumerate}  
\item for a suitable choice of a basis of $\overline{T}$, it yields a non-semisimple representation    
$$(\iota\circ \br_{F,p},\overline{T})\simeq \left(\begin{array}{cccc}
\ve_1 & \ast_{12} &  \ast_{13} & \ast_{14} \\
0 & \ve_2 &  \ast_{23} &  \ast_{24}\\
0 & 0 &  \ve_3 & \ast_{34} \\
0 & 0 &  0 & \ve_4 
\end{array}
\right)$$ 
where $\ve_{i},\ i=1,2,3,4$ are finite characters of $G_\Q$ and it 
satisfies either of the conditions:
\begin{enumerate}
\item if we write  
$$\br_{F,p}(\sigma)=\left(\begin{array}{cc}
\overline{A}(\sigma) & \overline{B}(\sigma)\\
\overline{C}(\sigma) & \overline{D}(\sigma) 
\end{array}
\right),\ \sigma\in G_\Q,\ \overline{C}\equiv 0$$
as a 2 by 2 block matrix representation whose each enrty is 2 by 2 matrix over $\F$, then $\overline{B}$ gives a 
non-rtivial extension of $\overline{A}$ by $\overline{D}$, and 
\item  if we write  
$$\br_{F,p}(\sigma)=\left(\begin{array}{cc}
\overline{A}_{1,1}(\sigma) & \overline{A}_{1,3}(\sigma)\\
\overline{A}_{3,1}(\sigma) & \overline{A}_{3,3}(\sigma) 
\end{array}
\right),\ \sigma\in G_\Q$$
where $\overline{A}_{i,j}$ is $i$ by $j$ matrix over $\F$, then $\overline{A}_{1,3}$ gives a 
non-rtivial extension of $\overline{A}_{1,1}$ by $\overline{A}_{3,3}$.
\end{enumerate}
\item The semi-simplification $(\iota\circ \br_{F,p})^{{\rm ss}}$ is isomorphic to  
$\ve^{2k-3}_p\oplus \ve^{k-1}_p\oplus \ve^{k-2}_p\oplus \textbf{1}$ where $\ve_p$ is the mod $p$ cyclotomic character of $G_\Q$ and 
$\textbf{1}$ stands for the trivial representation. 
\end{enumerate}
\end{theorem}
\begin{proof}Since $F\equiv G_k$ modulo $\mathcal{P}_K$, for any prime $\ell$ such that $\mathcal{P}_K\not| \ell$, 
we have the equality  $$\det(1-X\iota\circ \br_{F,p}({\rm Frob}_\ell))=(1-\ell^{2k-3}X)(1-\ell^{k-1}X)(1-\ell^{k-2}X)(1-X)$$
in $\F[X]$ (cf.  Corollary 1.4 of \cite{McCarthy}), $\F=\mathcal{O}_K/\mathcal{P}_K$ where $X$ is a variable. By Brauer-Nesbitt theorem  
$(\iota\circ \br_{F,p})^{{\rm ss}}\simeq \ve^{2k-3}_p\oplus \ve^{k-1}_p\oplus \ve^{k-2}_p\oplus \textbf{1}$, hence we have the second claim. 
This isomorphism guarantees the existence of a $G_\Q$ stable $\mathcal{O}_K$-lattice $T$ of $\iota\circ \rho_{F,p}$ such that 
$T$ yields an integral representation $\iota\circ \rho_{F,p}:G_\Q\lra {\rm GL}_4(\mathcal{O}_K)$ whose 
reduction $\iota\circ \br_{F,p}:G_\Q\lra {\rm GL}_4(\F)$ is upper triangular. 

Write 
$$\rho_{F,p}(\sigma)=\left(\begin{array}{cc}
A(\sigma) & B(\sigma)\\
C(\sigma) & D(\sigma) 
\end{array}
\right),\ \sigma\in G_\Q$$
as 2 by 2 block matrix representation whose each enrty is 2 by 2 matrix over $\mathcal{O}_K$. 
Let $\varpi$ be a uniformizer of $K$ and put $Q={\rm diag}(\varpi,\varpi,1,1)$. 
Then we see that $Q^{-1}\rho_{F,p}(\sigma)Q=\left(\begin{array}{cc}
A(\sigma) & \varpi^{-1}B(\sigma)\\
\varpi C(\sigma) & D(\sigma) 
\end{array}
\right)$. We apply this conjugation repeatedly until we obtain $\varpi {\not|} B(\sigma_0)$ for some $\sigma_0\in G_\Q$. Then By applying Ribet's well known 
argument in the proof of Proposition 2.1 of \cite{Ribet} we have the non-triviality of $\overline{B}$.   

Next we write 
$$\rho_{F,p}(\sigma)=\left(\begin{array}{cc}
A_{1,1}(\sigma) & A_{1,3}(\sigma)\\
A_{3,1}(\sigma) & A_{3,3}(\sigma) 
\end{array}
\right),\ \sigma\in G_\Q$$
where $A_{i,j}(\sigma)$ is $i$ by  $j$ matrix over $\mathcal{O}_K$. Put $Q_1={\rm diag}(\varpi,1,1,1)$. 
Then we see that $Q^{-1}_1\rho_{F,p}(\sigma)Q_1=\left(\begin{array}{cc}
A_{1,1}(\sigma) & \varpi^{-1}A_{1,3}(\sigma)\\
\varpi A_{3,1}(\sigma) & A_{3,3}(\sigma) 
\end{array}
\right)$. 
Then a similar argument shows the claim. 
\end{proof}
We are now ready to prove the second main result. 
\begin{proof}(A proof of Theorem \ref{main2}) 
Assume that $F$ is a Hecke eigen non-trivial $\mathcal{O}_K$-integral cusp form for some finite extension $K/\Q_p$ which is congruent to 
$G_k$ modulo $\mathcal{P}$. 
We first assume that $F$ is a Saito-Kurokawa lift. Then there exists a new form $f$ in $S_{2k-2}({\rm SL}_2(\Z))$ such that 
$F=SK(f)$ is a Saito-Kurokawa (or Maass) lift from $f$. 
The congruence condition implies that for Hecke eigenvalues. It follows from this that 
$a_\ell(f)\equiv \ell^{2k-3}+1$ mod $p$ for any prime $\ell\neq p$ and it yields that 
$(\br_{f,p})^{{\rm ss}}\sim \ve^{2k-3}_p\oplus \textbf{1}$ where $\br_{f,p}$ is the mod $p$ Galois representation 
attached to $f$. By Ribet's argument \cite{Ribet}  we have a non-trivial 
extension $0\lra \ve^{2k-3}_p\lra \ast \lra \textbf{1}\lra 0$ as $\F[G_\Q]$-modules which never splits. Here $\F$ is some 
finite extension of $\F_p$. 
Further it comes from a cusp form of level one and it implies  
${\rm ord}_p(\zeta(3-2k))>0$ by using well-known argument (see \cite{BCG} and Section 8 of \cite{Brown}).  

Let us suppose that $F$ is not any Saito-Kurokawa lift. As seen before it is a genuine form. 
By Theorem \ref{galois} there exists a $G_\Q$-stable $\mathcal{O}_K$-lattice $T$ such that the reduction $(\iota\circ \br_{F,p},\overline{T})$ takes the form 
\begin{equation}\label{form}
\iota\circ \br_{F,p}=\left(\begin{array}{cccc}
\ve_1 & \ast_{12} &  \ast_{13} & \ast_{14} \\
0 & \ve_2 &  \ast_{23} &  \ast_{24}\\
0 & 0 &  \ve_3 & \ast_{34} \\
0 & 0 &  0 & \ve_4 
\end{array}
\right).
\end{equation}
Since $p-1>2k-2$, any two characters among four characters are distinguished.  
Assume that $p$ does not divide $h^+_p\cdot \zeta(1-k)$. Recall that $H^1_{f}(G_\Q,\F(\ve^i_p))\simeq {\rm Hom}_{\F_p[G_\Q]}(A_i,\F)$ 
(cf. Section 1.6 of \cite{Rubin} for the Bloch-Kato Selmer groups $H^1_f$ and its relation to ideal class groups),
 and the facts that $A_0=A_1=0$ (see Proposition 6.16 of \cite{Wa}) and ${\rm ord}_p(B_{p-i})=0$ implies $A_i=0$ (cf. Theorem 6.17, Section 6 of \cite{Wa}).  
Hence 
\begin{equation}\label{va}
{\rm Ext}^1_{\F[G_\Q],MF_\F}(\F\ve_i,\F\ve_j)=H^1_{f}(G_\Q,\F(\ve_j\ve^{-1}_i))=0
\end{equation}
 except for $(\ve_i,\ve_j)\in \{(\ve^{2k-3}_p,\textbf{1}),\ (\textbf{1},\ve^{2k-3}_p),\ 
(\ve^{k-1}_p,\textbf{1}) \}$ where ${\rm Ext}^1_{\F[G_\Q],MF_\F}(\ast_1,\ast_2)$ is the group of all extension classes of 
$\ast_1$ by $\ast_2$ as  
$\F[G_\Q]$-modules whose restriction to $\F[G_{\Q_p}]$ comes from an object of $MF_\F$ which stands for the category of Fontaine-Laffile modules over $\F$. Since $p-1>2k-2$, 
we will use the fact that any extensions appear in (\ref{form}) belong to 
${\rm Ext}^1_{\F[G_\Q],MF_\F}(\ast_1,\ast_2)$ for some characters $\ast_1,\ast_2$ of $G_\Q$  
since $\iota\circ \br_{F,p}$ is the reduction of a $p$-adic Galois representation which is unramified outside $p$ and 
is crystalline at $p$ (see Lemma 7.7 of \cite{Brown}). 

We prove the claim by a brute force attack. The readers may consult Section 8 of \cite{Brown} for a related 
argument.  
First we assume that $\ast_{12}$ is non-trivial. 
This implies 
\begin{equation}\label{non-trivial-ext}
(\ve_1,\ve_2)\in \{(\ve^{2k-3}_p,\textbf{1}),\ (\textbf{1},\ve^{2k-3}_p),\ 
(\ve^{k-1}_p,\textbf{1}) \}.
\end{equation} 
If $(\ve_1,\ve_2)=(\ve^{2k-3}_p,\textbf{1})$ we have nothing to prove. 

Assume that $(\ve_1,\ve_2)=(\textbf{1},\ve^{2k-3}_p)$. Then $\ve_3,\ve_4\in \{\ve^{k-1}_p,\ve^{k-2}_p\}$. 
By (\ref{va}), we may assume $\ast_{34}=0$ and it yields $\ast_{2i}=0$ for $i=3,4$ since 
$\ast_{2i}\in H^1_f(G_\Q,\F\ve_i\ve^{-(2k-1)})=0$. 
Hence we may start with 
\begin{equation}\label{form1}
\iota\circ \br_{F,p}=\left(\begin{array}{cccc}
\textbf{1} & \ast_{12} &  \ast_{13} & \ast_{14} \\
0 & \ve^{2k-3}_p & 0 &  0\\
0 & 0 &  \ve_3 &0 \\
0 & 0 &  0 & \ve_4 
\end{array}
\right),\ \ve_3,\ve_4\in \{\ve^{k-1}_p,\ve^{k-2}_p\}.
\end{equation}

Put $s_0=\left(\begin{array}{cc}
0 & 1  \\
1 & 0
\end{array}
\right),\ S_0=\left(\begin{array}{cc}
s_0 & 0_2\\
0_2 & 1_2 
\end{array}
\right)$, and for $i=1,2,3,4$, let $P_i$ be the diagonal matrix over $\mathcal{O}_K$ of size 4 whose $(j,j)$-entry is $\varpi$ if $j=i$, 1 otherwise. Put $\sigma=\iota\circ \rho_{F,p}$ and $\overline{\sigma}:=\iota\circ \br_{F,p}$ for simplicity. 
Then 
\begin{equation}\label{form2}
S_0\overline{\sigma} S^{-1}_0=\left(\begin{array}{cccc}
\ve^{2k-3}_p & 0 & 0 & 0 \\
\ast_{12} & \textbf{1} &  \ast_{13} &  \ast_{14}\\
0 & 0 &  \ve_3 & 0 \\
0 & 0 &  0 & \ve_4 
\end{array}
\right).
\end{equation}
Observe the second column and law.  Consider $\sigma_1:=P_1S_0T(P_1S_0)^{-1}$ which is $\mathcal{O}$-integral and its reduction $\overline{\sigma}_1$. 
Then 
\begin{equation}\label{form3}
\overline{\sigma}_1=\left(\begin{array}{cccc}
\ve^{2k-3}_p & \ast_1 & 0 & 0 \\
0 & \textbf{1} &  0 &  0\\
0 & \ast_3 &  \ve_3 & 0 \\
0 & \ast_ 4&  0 & \ve_4 
\end{array}
\right).
\end{equation}
Then each $\ast_i$ ($i=3,4$) gives an element of $H^1_f(G_\Q,\F\ve^{-1}_i)$. However it is zero by (\ref{va}).  
This implies 
\begin{equation}\label{form4}
\overline{\sigma}_1=\left(\begin{array}{cccc}
\ve^{2k-3}_p & \ast_1 & 0 & 0 \\
0 & \textbf{1} &  0 &  0\\
0 & 0 &  \ve_3 & 0 \\
0 & 0&  0 & \ve_4 
\end{array}
\right).
\end{equation}
By Theorem \ref{galois}-(1)-(b), $\ast_1$ has to be non-trivial. Hence we have the claim. 

Assume that $(\ve_1,\ve_2)=(\textbf{1},\ve^{k-1}_p)$ and $\ve_3=\ve^{2k-3}_p,\ve_4=\ve^{k-2}_p$. 
Then by (\ref{va}) we may start with  
\begin{equation}\label{form11}
\overline{\sigma}=\left(\begin{array}{cccc}
\textbf{1} & \ast_{12} &  \ast_{13} & 0 \\
0 & \ve^{k-1}_p & 0 &  0\\
0 & 0 &  \ve^{2k-3}_p &0 \\
0 & 0 &  0 & \ve^{k-2}_p 
\end{array}
\right).
\end{equation}
Note that we first observe $\ast_{34}=0$, and then $\ast_{i4}=0$ for $i=1,2$ and also $\ast_{23}=0$ accordingly. 
Then we have 
\begin{equation}\label{form12}
\overline{\sigma}=\left(\begin{array}{cccc}
\ve^{k-1}_p & 0 &  0 & 0 \\
\ast_{12} & \textbf{1} & \ast_{13} &  0\\
0 & 0 &  \ve^{2k-3}_p &0 \\
0 & 0 &  0 & \ve^{k-2}_p 
\end{array}
\right).
\end{equation}
As in (\ref{form3}) we proceed as 
\begin{equation}\label{form13}
\overline{\sigma}_1=\left(\begin{array}{cccc}
\ve^{k-1}_p & \ast_1 &  0 & 0 \\
0 & \textbf{1} & 0 &  0\\
0 & \ast_3 &  \ve^{2k-3}_p &0 \\
0 & \ast_4 &  0 & \ve^{k-2}_p 
\end{array}
\right).
\end{equation}
By assumption, $\ast_1=0$ and $\ast_4=0$. Put $S_2=\left(\begin{array}{cccc}
1 & 0 &  0 & 0 \\
0 & 0 & 1 &  0\\
0 & 1 &  0 &0 \\
0 &0 &  0 & 1 
\end{array}
\right)$. Then we have $S_2\overline{\sigma}_1S^{-1}_2=\left(\begin{array}{cccc}
\ve^{k-1}_p & 0 &  0 & 0 \\
0 & \ve^{2k-3}_p & \ast_3 &  0\\
0 & 0 &  \textbf{1} &0 \\
0 & 0 &  0 & \ve^{k-2}_p 
\end{array}
\right)$. By Theorem \ref{galois}-(1)-(a), $\ast_3$ has to be non-trivial. Hence we have the claim. 
A similar argument deduce the same claim when  $(\ve_1,\ve_2)=(\textbf{1},\ve^{k-1}_p)$ and $\ve_4=\ve^{2k-3}_p,\ve_3=\ve^{k-2}_p$ and therefore the details are omitted. 

What we have remained is the case when $\ast_{12}=0$. This case is also handled similarly and even simpler than the previous case. We omit the details but notice that we will use Theorem \ref{galois}-(1)-(a) occasionally. 
\end{proof}

\section{A proof of Theorem \ref{main3}}
In this section we give a sketch of a proof of  Theorem \ref{main3}. Let us keep the notation being there. 
Let $F$ be the Siegel cusp form in the statement and $\rho_{F,p}$ (resp. $\br_{F,p}$) be its $p$-adic 
(resp. mod $p$) Galois representation (\ref{galois-const}).  
Since $p-1>2k-3$ and $\iota\circ \rho_{F,p}$ is crystalline at $p$, we have the complete classification of 
$\iota\circ \rho_{F,p}|_{G_{\Q_p}}$ (see Proposition 3.1 of \cite{Dieulefait}). 
It follows from this that 
$$(\iota\circ \br_{F,p})^{{\rm ss}}\simeq \ve^{2k-3}_p\oplus \ve^{k-1}_p\oplus \ve^{k-2}_p\oplus \textbf{1}.$$
It is obvious that the proof of Theorem \ref{main2} works similarly in this case. Only thing we have to change is 
to replace (\ref{non-trivial-ext}) with 
$$
(\ve_1,\ve_2)\in \{(\ve^{2k-3}_p,\textbf{1}),\ (\textbf{1},\ve^{2k-3}_p),\ 
(\ve^{k-2}_p,\textbf{1}) \}.
$$
Therefore we omit the details. 
\section{Examples}Let $k,p,\chi$ be as in Theorem \ref{main-SK}. In this section we list a few examples of $(k,p,\chi)$ such that $p-1>2k-2$ and ${\rm ord}_p(\ds\frac{B_{2k-2}}{2k-2})>0$ 
but ${\rm ord}_p(\ds\frac{B_{k-1},\chi_D}{k-1})=0$ for some imaginary 
quadratic character $\chi=\chi_D$ with a fundamental discriminant $D$. 
Note that $H(k-1,|D|)=-\ds\frac{B_{k-1},\chi_D}{k-1}$. 

\begin{table}[htbp]
\label{tab1}
\begin{center}
\begin{tabular}{|c|c|c|c|c|c|c|}
\hline
$k$ & 10 &  12 & 14 & 16 & 18 & 20 \\
\hline
$2k-2$ & 18 &  22 & 26 & 30 & 34 & 38 \\
\hline 
$p$ & 43867 & 131, 593 & 657931 & 1721, 1001259881 & 151628697551 & 154210205991661 \\
\hline 
$\chi$ & $\chi_{-4}$ & $\chi_{-4}$ & $\chi_{-4}$ & $\chi_{-4}$ & $\chi_{-4}$ & $\chi_{-4}$ \\
\hline
$t$ & 1 &  1 & 1 & 1 & 1 & 1 \\
\hline 

\end{tabular}
\end{center}
\caption{}
\end{table}
Each column except for the entry in the bottom row stands for such a triple $(k,p,\chi)$. 
The bottom line shows the number of Saito-Kurokawa lifts in Theorem \ref{main-SK}. 
In any case in the table, the prime $p$ does not divide the discriminant of the Hecke field for $S_{2k-2}({\rm SL}_2(\Z))$ and 
it yields $t=1$.

\end{document}